\documentclass[a4, 12pt]{amsart}
\UseRawInputEncoding
\usepackage{mathrsfs}
\usepackage{amsmath}
\usepackage{amsfonts}
\usepackage{amssymb}
\usepackage{amsmath,amssymb,amsthm}
\usepackage{amsmath,amssymb,amsthm,amscd}
\usepackage[frame,cmtip,arrow,matrix,line,graph,curve]{xy}
\usepackage{graphpap}
\usepackage{color,xcolor}
\usepackage[mathscr]{eucal}
\usepackage{verbatim}
\usepackage[colorlinks, linkcolor=red, anchorcolor=blue,citecolor=blue]{hyperref}
\usepackage{cite}

\numberwithin{equation}{section} \numberwithin{equation}{section}
\newtheorem{thm}{Theorem}[section]
\newtheorem{lem}{Lemma}[section]

\newtheorem{prop}{Proposition}[section]
\newtheorem*{CZ}{Theorem (Chen and Zeng's Universal Inequality)}

\newtheorem{rem}{Remark}[section]

\newtheorem*{ack}{Acknowledgment}
\newtheorem*{problem}{Problem}
\newcommand{\DOI}[1]{doi: \href{https://doi.org/#1}{#1}}

\renewcommand{\oddsidemargin}{5mm}

\makeatletter
\title[Universal Bounds for Fractional Laplacian]{Universal Bounds for Fractional Laplacian\\ on a Bounded Open Domain in $\mathbb{R}^{n}$}
\author[L. Zeng ]{Lingzhong Zeng }

\address{Lingzhong Zeng
\\  \newline \indent Jiangxi Provincial Center for Applied Mathematics$^{1}$
\\  \newline \indent   Jiangxi Normal University, Nanchang 330022,  China.
\\  \newline \indent  School of Mathematics and Statistics$^{2}$
\\  \newline \indent   Jiangxi Normal University, Nanchang 330022,  China.
\\  \newline \indent lingzhongzeng@yeah.net }

\begin{document}
\maketitle

\begin{abstract}   Let $\Omega$ be a bounded open domain on the Euclidean space $\mathbb{R}^{n}$ and $\mathbb{Q}_{+}$ be the set of all positive rational numbers. In 2017, Chen and Zeng investigated  the eigenvalues with higher order of the fractional Laplacian $\left.(-\Delta)^{s}\right|_{\Omega}$ for $s>0$ and $s \in \mathbb{Q}_{+}$,  and they obtained a universal  inequality of Yang type(\emph{ Universal inequality and upper bounds of eigenvalues
for non-integer poly-Laplacian on a bounded domain,  Calculus of Variations and Partial Differential Equations,
 (2017) \textbf{56}:131}).  In the
spirit of Chen and Zeng's work, we study the eigenvalues of fractional Laplacian, and establish an inequality of eigenvalues with lower order  under the same condition. Also, our eigenvalue inequality is  universal and  generalizes the eigenvalue inequality for the poly-harmonic operators given by  Jost et al.(\emph{Universal bounds for eigenvalues of polyharmonic operator.   Trans.  Amer. Math.  Soc. {\bf 363}(4), 1821-1854 (2011)}).
\end{abstract}


\footnotetext{The author  partially supported by the  National Natural Science Foundation of China (Grant Nos. 11861036 and 11826213).}

\section{Introduction}
Let $\Omega$ be a bounded domain with piecewise smooth boundary
$\partial\Omega$ on an $n$-dimensional Euclidean space
$\mathbb{R}^{n}$. The Dirichlet eigenvalue problem of the poly-Laplacian with
any
order is described by \begin{equation}\label{poly-Laplace-problem}\begin{cases}(-\Delta)^{l}u = \lambda u,\ \ & {\rm in}\ \ \Omega,\\[2mm]
u=\dfrac{\partial u}{\partial\nu}=\cdots=\dfrac{\partial^{l-1}
u}{\partial\nu^{l-1}}=0,\ \ & {\rm on}\ \ \partial\Omega,\end{cases}\end{equation}
where $\Delta$ is the Laplacian and $\nu$ denotes the outward unit
normal vector field of $\partial\Omega$. It is well known that the spectrum of the eigenvalue problem \eqref{poly-Laplace-problem} is
discrete and satisfies
\begin{equation*}
0<\lambda_{1}\leq\lambda_{2}\leq\cdots\leq\lambda_{k}\leq\cdots\rightarrow+\infty,
\end{equation*}
where  each eigenvalue is repeated according to its multiplicity  and $\lambda_{k}$ denotes the $k^{th}$ eigenvalue.  When $l=1$, eigenvalue problem \eqref{poly-Laplace-problem} is
called a fixed membrane problem. For this case, in 1956, Payne et al. \cite{PPW} (Thompson-Colin \cite{T} (for all $n \geq 3$)) proved a universal inequality as follows:

\begin{equation}\label{ppw-ineq}\lambda_{k+1}-\lambda_{k}\leq\frac{4}{nk}\sum^{k}_{i=1}\lambda_{i}.\end{equation}
Furthermore, in various settings, many mathematicians extended the above universal inequality.  In particular, Hile and Protter \cite{HP} proved the following universal inequality of eigenvalues:

\begin{equation}\label{hp-ineq}\sum^{k}_{i=1}\frac{\lambda_{i}}{\lambda_{k+1}-\lambda_{i}}\geq\frac{nk}{4},\end{equation}
which is sharper than inequality \eqref{ppw-ineq}. Furthermore, Yang \cite{Y} made an amazing contribution to eigenvalue problem \eqref{poly-Laplace-problem} (cf. \cite{CY1}), and obtained a very sharp universal bound as follows:

\begin{equation}\label{y1-ineq}\sum^{k}_{i=1}(\lambda_{k+1}-\lambda_{i})^{2}\leq\frac{4}{n}\sum^{k}_{i=1}(\lambda_{k+1}-\lambda_{i})\lambda_{i}.\end{equation}
Applying \eqref{y1-ineq} and  a celebrated recursion formula, Cheng and Yang gave a sharp upper bounds for the eigenvalues with respect to $k^{\frac{2}{n}}$ ( cf. \cite{CY3}). In addition, for the case of  degenerate elliptic operators, we refer the readers to \cite{CQLX,CP} and references therein.

On the other hand, Payne et al. \cite{PPW} investigated the lower order eigenvalues
and proved the following universal inequality in 1956:

\begin{equation}\label{1.2}\lambda_{2}+\lambda_{3}\leq6\lambda_{1},\end{equation}
for $\Omega\subset\mathbb{R}^{2}$,
and further proposed a famous conjecture for
$\Omega\subset\mathbb{R}^{n}$ as follows:

\begin{equation}
\frac{\lambda_{2} +\lambda_{3} +\cdots+ \lambda_{n+1}}{
\lambda_{1}}\leq n\frac{\lambda_{2}(\mathbb{B}^{n})}{\lambda_{1}(\mathbb{B}^{n})},\end{equation}where $\lambda_{i}(\mathbb{B}^{n})(i=1,2)$ denotes the $i^{th}$ eigenvalue of Laplacian on the ball $\mathbb{B}^{n}$ with the same volume as the bounded domain $\Omega$, i.e., ${\rm Vol}(\Omega)={\rm Vol}(\Omega^{\ast})$. Attacking this conjecture, Brands \cite{Bran} improved \eqref{1.2} to the following: $\lambda_{2}+\lambda_{3}  \leq\lambda_{1}(3 + \sqrt{7}),
$ when $n=2$.
Furthermore, Hile and Protter \cite{HP} proved
$\lambda_{2} +\lambda_{3}
\leq 5.622\lambda_{1}.$ In 1980, Marcellini \cite{Mar} obtained $\lambda_{2}+\lambda_{3} \leq(15 + \sqrt{345})/6\lambda_{1}.$ In 2011, Chen and Zheng \cite{CZ13} proved $\lambda_{2}+\lambda_{3} \leq5.3507\lambda_{1}.$ For general case,  Ashbaugh and Benguria
\cite{AB4} established an interesting universal inequality as follows:

\begin{equation}\label{1.16}
\frac{\lambda_{2} +\lambda_{3} +\cdots+ \lambda
_{n+1}}{
\lambda_{1}}\leq n + 4,\end{equation} for $\Omega\subset\mathbb{R}^{n}$, in 1993. For more references on the solution of
this conjecture, we refer the reader to
\cite{AB2,AB3,HP,Mar,Sun,SCY}.

In the case of $l = 2$, problem \eqref{poly-Laplace-problem} is also called a clamped plate problem of bi-harmonic operator. There are similar inequalities as to \eqref{ppw-ineq}, \eqref{hp-ineq} and \eqref{y1-ineq}, which were studied in
\cite{Ash2,CQ,CY2,Hook1,Hook2,HY,WX}, respectively. For any positive integer $l$, more results can be found in \cite{CQ,Hook1,WC}.
For this case, there is also an analogue of Yang type inequality for the poly-Laplacian operator

\begin{equation}\sum_{i=1}^{k}\left(\lambda_{k+1}-\lambda_{i}\right)^{2} \leq \frac{4 l(2 l+n-2)}{n^{2}} \sum_{i=1}^{k} \lambda_{i}\left(\lambda_{k+1}-\lambda_{i}\right),\end{equation}
which was obtained by Cheng et al. in \cite{CIM2} in 2009.
In 1998, Ashbaugh \cite{Ash2} announced two interesting inequalities without proofs as follows:
\begin{equation}\label{1.4} \sum^{n}_{i=1}(\lambda^{\frac{1}{2}}
_{i+1}-\lambda^{\frac{1}{2}}_{1})\leq4\lambda^{\frac{1}{2}}_{1},
\end{equation}and

\begin{equation}\label{1.5}\sum^{n}_{i=1} (\lambda_{i+1}-\lambda_{1})\leq24\lambda_{1}.\end{equation}
The proofs of \eqref{1.4}) and \eqref{1.5}) were given by Cheng et al. in \cite{CIM2} in 2009.
In fact, they considered more general case and proved

\begin{equation}\label{1.8}\sum^{n}_{i=1}(\lambda^{\frac{1}{l}}_{i+1}-\lambda_{1}^{\frac{1}{l}})^{l-1}\leq(2l)^{l-1}\lambda^{\frac{l-1}{l}}_{1},
\end{equation} for $l\geq2$, and

\begin{equation}\label{1.9}\sum^{n}_{i=1}(\lambda_{i+1}-\lambda_{1})\leq4l(2l-1)\lambda_{1},\end{equation}for any $l=1,2,\cdots$.
It is easy to find that \eqref{1.8} and \eqref{1.9}
become \eqref{1.4} and \eqref{1.5} when $l=2$, respectively. Moreover,
\eqref{1.9} covers \eqref{1.16} when $l=1$. In 2010, Cheng et al.  \cite{CHW}
proved

\begin{equation}\label{1.7}\sum^{n}_{i=1}(\lambda_{i+1}-\lambda_{1})^{\frac{1}{2}}\leq[8(n+2)\lambda_{1}]^{\frac{1}{2}},
\end{equation}  which is generalized by Sun and the author  in \cite{SZ}. In 2011, Jost et al.  \cite{JJWX} derived the following inequality

\begin{equation}\label{1.10}\sum^{n+1}_{i=2}\lambda_{i}+\sum^{n-1}_{i=1}\frac{2(l-1)i}{2l+i-1}(\lambda_{n+1-i}-\lambda_{1})\leq(n+4l(2l-1))\lambda_{1},
\end{equation} which covers \eqref{1.16} when $l=1$ and improves
\eqref{1.5} when $l = 2$.

Furthermore, we consider the fractional Laplace operator restricted to $\Omega$ and denote it by $\left.(-\Delta)^{s}\right|_{\Omega}$ with $s>0,$ which is defined as the pseudo-differential operator restricted to $\Omega$. In other words, the fractional Laplacian can be defined by
\[
(-\Delta)^{s} u(x)=:\textbf{ P.V.} \int_{\mathbb{R}^{n}} \frac{u(x)-u(z)}{|x-z|^{n+2s}} dz,
\]
where \textbf{P.V.} denotes the principal value and $u:\mathbb{R}^{n} \rightarrow \mathbb{R}$. We note that fractional Laplacian is not a local operator when $s$ is not a integer. Define the characteristic function $\chi_{\Omega}: t \mapsto \chi_{\Omega}(t)$ by $\chi_{\Omega}(t)=1$ when $x \in \Omega$, while $\chi_{\Omega}(t)=0$ when $ x \in \mathbb{R}^{n} \backslash \Omega$, then the special pseudo-differential operator can be represented as the Fourier transform of the function $u$ \cite{Lan}, namely
\[
\left.(-\Delta)^{s}\right|_{\Omega} u:=\chi_{\Omega}\mathscr{F}^{-1}\left[|y|^{2s} \mathscr{F}\left[u\right]\right],
\]
where $\mathscr{F}[u]$ denotes the Fourier transform of a function $u: \mathbb{R}^{n} \rightarrow \mathbb{R}$
\[
\mathscr{F}[u]( y )=\hat{u}( y )=\frac{1}{(2 \pi)^{n/2}} \int_{ \mathbb{R}^{n}} e^{-\sqrt{-1} x \cdot  y } u(x) d x,
\]  and the notation $\mathscr{F}^{-1}$ is the inverse of Fourier transform. Here, $\mathscr{F}^{-1}$ is given by

\[
\mathscr{F}^{-1}[w](x)=\breve{w}(x)=\int_{ \mathbb{R}^{n}} e^{\sqrt{-1} x \cdot  y } w(y) d y.
\]
It is well known that the fractional Laplacian operator $(-\Delta)^{s}$ can be considered as the infinitesimal generator of the symmetric $s$-stable process \cite{Bl,Ba1,Ba2,Ba3,BaY}. Suppose that a stochastic process $X _{t}$ has stationary independent increments and its transition density (i.e., convolution kernel) $p^{s}(t, x, z)=p^{s}(t, x-z), t>0, x, z \in \mathbb{R}^{n}$ is determined by the following Fourier transform
\[
\operatorname{Exp}\left(-t|y|^{2s}\right)=\int_{\mathbb{R}^{n}} e^{\sqrt{-1} y  \cdot z} p^{2s}(t,z)dz,\quad t>0,\quad y \in \mathbb{R}^{n},
\]
then we can say that the process $X _{t}$ is an $n$-dimensional symmetric $s$-stable process with order $s\in(0,1]$ in $\mathbb{R}^{n}$( also see \cite{Ba1,Ba2,BaY})
There are two important examples of this process: one is the Cauchy process; and the other is the Brownian motion.
The fractional Laplacian operator $(-\Delta)^{s}$ is the Riesz-Feller derivative of fractional order $2 s>0$, i.e. $(-\Delta)^{s}=\frac{\partial^{2 s}}{\partial|x|^{2 s}}$,  which appears in a wide class of physical systems including L\'{e}vy flights and stochastic interfaces.  More detail information on L\'{e}vy process
will be found in a textbook \cite{Sat} and a good survey \cite{App}. We also refer the readers to an excellent article \cite{CZ}, where some important and interesting literatures on fractional Laplacian have been given by Chen and Zeng.
We assume that $\phi \in C_{0}^{\infty}(\Omega) \subset C_{0}^{\infty}\left(\mathbb{R}^{n}\right),$ then the quadratic form $\left\langle\left.(-\Delta)^{s}\right|_{\Omega} \phi, \phi\right\rangle=\int_{\Omega} \bar{\phi}(-\Delta)^{s} \phi d x$
is positive on a dense subset of $L^{2}(\Omega),$ and it can be extended to a unique minimal positive operator on $L^{2}(\Omega) .$ In addition, the fractional Laplacian operator $\left.(-\Delta)^{s}\right|_{\Omega}$ is self-adjoint with purely discrete spectrum. Specially, there is a sequence of discrete eigenvalues which can be ordered, after counting multiplicity, as

$$
0<\lambda_{1}<\lambda_{2} \leq \cdots \leq \lambda_{k} \leq \cdots \rightarrow+\infty.
$$Comparing with the case of Laplacian, the eigenvalue of fractional Laplacian $\left.(-\Delta)^{s}\right|_{\Omega} .$ is much less studied, but it has recently received
more attention. Here, we only refer the reader to papers \cite{BaY,CZ,HY1,WSZ,YY1,YY2} and the references therein.
In particular, Chen and Zeng investigated the eigenvalues with higher order of fractional Laplacian. See \cite{CZ}. Applying an important result due to Ilias and Makhoul \cite{IM} and a fine properties of commutators for some
operators, they proved an interesting result in \cite{CZ} in 2017 as follows:
\begin{CZ}\label{thm-cz} Let $\Omega$ be a bounded open domain in $\mathbb{R}^{n}$. Suppose that $\left\{\lambda_{i}\right\}_{i \geq 1}$ are eigenvalues of operator $\left.(-\Delta)^{s}\right|_{\Omega},$ where $s=\frac{1}{m}$ with the positive integer $m \geq 2$ or $s>\frac{1}{2}$ and $s \in \mathbb{Q}_{+} .$ Then,

\begin{equation*}\begin{aligned}
\left[\sum_{i=1}^{k}\left(\lambda_{k+1}-\lambda_{i}\right)^{2}\right]^{2}
\leq \frac{4 s(n+2 s-2)}{n^{2}}\left(\sum_{i=1}^{k}\left(\lambda_{k+1}-\lambda_{i}\right)^{2} \lambda_{i}^{\frac{s-1}{s}}\right)\left(\sum_{i=1}^{k}\left(\lambda_{k+1}-\lambda_{i}\right) \lambda_{i}^{\frac{1}{s}}\right).
\end{aligned}\end{equation*}\end{CZ}By the above theorem and a variant of Chebyshev sum inequality, Chen  and Zeng \cite{CZ} deduced
an eigenvalue inequality of Yang type as follows:
\begin{equation}\sum_{i=1}^{k}\left(\lambda_{k+1}-\lambda_{i}\right)^{2} \leq \frac{4 s(n+2 s-2)}{n^{2}} \sum_{i=1}^{k}\left(\lambda_{k+1}-\lambda_{i}\right) \lambda_{i},\end{equation}
which implies an upper bounds  for the eigenvalues:
\begin{equation*}\lambda_{k+1} \leq\left(1+\frac{4 s(n+2 s-2)}{n^{2}}\right) k^{\frac{2 s(n+2 s-2)}{n^{2}}} \lambda_{1}.\end{equation*}
As the above argument, Chen and Zeng only discussed the eigenvalues with higher order for the fractional Laplacian. However, they have not addressed the case of lower order. In this regard, it is natural to propose the following problem:

\begin{problem}
Can one obtain a universal inequality for the eigenvalues with lower order of the
fractional Laplacian on  $\mathbb{R}^{n}$?\end{problem}
Under the same assumption as Chen and Zeng's Theorem, we answer the above problem. In details, we establish a key lemma in section \ref{sec2}, and by utilizing this lemma, obtain an eigenvalue inequality in section \ref{sec3}. This is what the following theorem says.

\begin{thm} \label{thm1.1}\textbf{\emph{(Universal inequality)}} Let $\Omega$ be a bounded open domain in $\mathbb{R}^{n}$. Suppose that $\mathbb{Q}_{+}$ is the set of all positive rational numbers and $\left\{\lambda_{i}\right\}_{i \geq 1}$ are eigenvalues of operator $\left.(-\Delta)^{s}\right|_{\Omega},$ where $s=\frac{1}{m}$ with the positive integer $m \geq 2$ or $s>\frac{1}{2}$ and $s \in \mathbb{Q}_{+} .$ Then, we have

\begin{equation}\label{thm1.1-ineq}\sum_{i=1}^{n}\left(\lambda_{i+1}-\lambda_{1}\right)^{\frac{1}{2}} \leq[4s(n+2s-2)]^{\frac{1}{2}} \lambda_{1}^{\frac{1}{2}}.\end{equation}
\end{thm}
\begin{rem}Since the constant $[4s(n+2s-2)]^{\frac{1}{2}} $ in \eqref{thm1.1-ineq} is
not dependent on $\Omega$, eigenvalue inequality \eqref{thm1.1-ineq} is universal.\end{rem}

\section{A key lemma and its proof}\label{sec2}
In this section, we shall prove a key lemma, which plays an important role in the proof of Theorem \ref{thm1.1}. Firstly, let us  explain the operation rules of the
fractional Laplacian in a bounded domain. Let $u$ and $w$ is two functions defined on $\Omega\subset\mathbb{R}^{n}$, and we recall that the fractional Laplacian is defined by

\[
\left.(-\Delta)^{s}\right|_{\Omega} u:=\chi_{\Omega}\mathscr{F}^{-1}\left[|y|^{2s} \mathscr{F}\left[u\right]\right],
\]
where
\[
\mathscr{F}[u]( y )=\hat{u}( y )=\frac{1}{(2 \pi)^{n/2}} \int_{ \mathbb{R}^{n}} e^{-\sqrt{-1} x \cdot  y } u(x) d x,
\]and $\mathscr{F}^{-1}$ is the inverse of Fourier transform $\mathscr{F}$, which is defined by \[
\mathscr{F}^{-1}[w](x)=\widehat{w}(x)=\int_{ \mathbb{R}^{n}} e^{\sqrt{-1} x \cdot  y } w(y) d y.
\] In 2014, Chen proved the following  Green formula. See \cite[Lemma 2.2]{CV}. Here, we can also conclude it by an alternative method.  Those two different techniques are based on the equivalent definitions of fractional Laplacian: the former is the principal value of integral, while the later is Fourier transformation.

\begin{prop}\textbf{\emph{(Green Formula)}}\label{prop-green-for} Let $w_{i}:\Omega(\subset\mathbb{R}^{n})\rightarrow\mathbb{R}$ be two functions satisfying $w_{i}\in C^{\infty}(\Omega)$ and $w_{i}(x)=0$ when $x\in\mathbb{R}^{n}\backslash\Omega$, where $i=1,2$. Then,
\begin{equation}\begin{aligned}\label{green-for}
 \int_{\Omega}w_{1}(x)\left.(-\Delta)^{s}\right|_{\Omega} w_{2}(x)dx =\int_{\Omega}w_{2}(x)\left.(-\Delta)^{s}\right|_{\Omega} w_{1}(x)dx.
\end{aligned}\end{equation}\end{prop}

\begin{proof}A straightforward calculation shows that
\begin{equation*}\begin{aligned}
&\quad\int_{\Omega}w_{1}(x)\left.(-\Delta)^{s}\right|_{\Omega} w_{2}(x)dx\\
&=\int_{\Omega}\chi_{\Omega}w_{1}  (x)\mathscr{F}^{-1}\left[|y|^{2s} \mathscr{F}\left[w_{2}(x)\right]\right]dx
\\
&=\frac{1}{(2 \pi)^{n/2}}\int_{\Omega}(\chi_{\Omega} w_{1})  (z)\mathscr{F}^{-1}\left[ \int_{\mathbb{R}^{n}} |y|^{2s}e^{-\sqrt{-1} x \cdot  y } w_{2}(x) d x\right] dz\\
&=\frac{1}{(2 \pi)^{n/2}}\int_{\Omega} w_{1}  (z)\mathscr{F}^{-1}\left[ \int_{\mathbb{R}^{n}} |y|^{2s}e^{-\sqrt{-1} x \cdot  y } w_{2}(x) d x\right] dz,
\end{aligned}\end{equation*}since $w_{1}\in C^{\infty}(\Omega)$ and $w_{1}(x)=0$ when $x\in\mathbb{R}^{n}\backslash\Omega$.
Furthermore, the definition of inverse of Fourier transform implies that
\begin{equation*}\begin{aligned}
&\quad\int_{\Omega}w_{1}(x)\left.(-\Delta)^{s}\right|_{\Omega} w_{2}(x)dx\\
&=\frac{1}{(2 \pi)^{n/2}}\int_{\Omega}w_{1}  (z)\left\{\int_{\mathbb{R}^{n}}e^{\sqrt{-1} z \cdot  y }\left[ \int_{\mathbb{R}^{n}} |y|^{2s}e^{-\sqrt{-1} x \cdot  y } w_{2}(x) d x\right] dy\right\}dz\\
&=\frac{1}{(2 \pi)^{n/2}}\int_{\Omega}\left\{\int_{\mathbb{R}^{n}}\int_{\mathbb{R}^{n}}w_{1}  (z) e^{\sqrt{-1} z \cdot  y } |y|^{2s}e^{-\sqrt{-1} x \cdot  y } w_{2}(x) d x  dy \right\}dz\\&=\frac{1}{(2 \pi)^{n/2}}\int_{\Omega}\left\{\int_{\mathbb{R}^{n}}\int_{\mathbb{R}^{n}}w_{1}  (z) e^{-\sqrt{-1} z \cdot  y } |y|^{2s}e^{\sqrt{-1} x \cdot  y } w_{2}(x) d x  dy \right\}dz\\
&=\int_{\mathbb{R}^{n}} w_{2}(x)\left\{\int_{\mathbb{R}^{n}}|y|^{2s}e^{\sqrt{-1} x \cdot  y }\left[ \frac{1}{(2 \pi)^{n/2}}\int_{\mathbb{R}^{n}}e^{-\sqrt{-1} z \cdot  y } w_{1}  (z) d z\right] dy\right\}dx\\
&=\int_{\mathbb{R}^{n}}(\chi_{\Omega}w_{2})(x)\mathscr{F}^{-1}\left[|y|^{2s} \mathscr{F}\left[w_{1}  (x)\right]\right]dx\\
&=\int_{\mathbb{R}^{n}} w_{2}(x)\left.(-\Delta)^{s}\right|_{\Omega} w_{1}(x)dx\\
&=\int_{\Omega}w_{2}(x)\left.(-\Delta)^{s}\right|_{\Omega} w_{1}(x)dx,
\end{aligned}\end{equation*}which gives \eqref{green-for}. Thus, we finish the proof of Proposition \ref{prop-green-for}.
\end{proof}

\begin{lem}\label{lem2.1}Suppose that $\{x_{1}, x_{2},\cdots, x_{n}\}$ is an arbitrary Euclidean coordinate. Let $\lambda_{j}$ and $u_{j}$ denote the $j$ -th eigenvalue and the corresponding normalized eigenfunction of $\left.(-\Delta)^{s}\right|_{\Omega},$ respectively. Assume that
\begin{equation}\label{ortho}\int_{\Omega}h_{i}u_{1}u_{j+1}dv=0, \ \ for \ \ 1\leq j<i\leq n,
\end{equation} where $h_{i}\in C^{\infty}(\Omega)$. Then,

\begin{equation}\label{lem-ineq}\sum^{n}_{i=1}(\lambda_{i+1}-\lambda_{1})^{\frac{1}{2}}\int_{\Omega}\Theta_{1}(h_{i})u_{1}dv
\leq2\left\{\sum^{n}_{i=1}\int_{\Omega}\Theta_{2}(h_{i})u_{1}dv\sum^{n}_{i=1}\int_{\Omega}\Psi_{i} dv\right\}^{\frac{1}{2}},\end{equation} where

\begin{equation}\label{Theta-1}\Theta_{1}(h_{i})=\frac{\partial}{\partial x_{i}}(h_{i}u_{1})-h_{i}\frac{\partial}{\partial x_{i}}u_{1},\end{equation}

\begin{equation}\label{Theta-2}\Theta_{2}(h_{i})=\frac{1}{2}\left[2 h_{i}(-\Delta)^{s}|_{\Omega}(h_{i}u_{1})-h^{2}_{i}(-\Delta)^{s}|_{\Omega}u_{1}-(-\Delta)^{s}|_{\Omega}(h_{i}^{2} u_{1})\right] ,\end{equation}

\begin{equation}\label{Psi}\Psi_{i}=u_{1,i}^{2},\end{equation}and
$$u_{1,i}=\frac{\partial u_{1}}{\partial x_{i}}.$$

\end{lem}

\begin{proof}Define function $\psi_{i}$ as follows:

\begin{equation*}\psi_{i}=h_{i}u_{1}-\tau_{i}u_{1},\end{equation*}
where $$\tau_{i}=\int_{\Omega}h_{i}u_{1}^{2}dv. $$According to
\eqref{ortho}, one can conclude that

\begin{equation}\label{int-phi-0}\int_{\Omega}\psi_{i}u_{j+1}dv=0,\ \  for\ \
0\leq j<i\leq n.\end{equation} Moreover, it follows from \eqref{int-phi-0} that

\begin{equation}\label{2.5}
\int_{\Omega}\psi_{i}h_{i}u_{1}dv=\int_{\Omega}\psi_{i}^{2}dv.\end{equation}
Since  $\psi_{i}$ satisfies \eqref{int-phi-0},  by Rayleigh-Ritz inequality, we know that

\begin{equation}\label{2.6}\lambda_{i+1}\leq\frac{\int_{\Omega}
\psi_{i}(-\Delta)^{s}|_{\Omega}\psi_{i}dv}{\int_{\Omega}\psi_{i}^{2}dv}.
\end{equation} Utilizing \eqref{int-phi-0} and \eqref{2.5}, we obtain

\begin{equation}\begin{aligned}\label{2.7}&\int_{\Omega}
\psi_{i}(-\Delta)^{s}|_{\Omega}\psi_{i}dv\\&=\int_{\Omega}\Big[(-\Delta)^{s}|_{\Omega}(h_{i}u_{1})-h_{i}(-\Delta)^{s}|_{\Omega}u_{1}\Big]\psi_{i}
dv+
\int_{\Omega} h_{i}\psi_{i}(-\Delta)^{s}|_{\Omega}u_{1}dv\\&=
\int_{\Omega}\Big[(-\Delta)^{s}|_{\Omega}(h_{i}u_{1})-h_{i}(-\Delta)^{s}|_{\Omega}u_{1}\Big]\psi_{i}
dv+\lambda_{1}\int_{\Omega}\psi_{i}^{2}dv
.\end{aligned}\end{equation} Substituting \eqref{2.7} into \eqref{2.6}, we
yield

\begin{equation}\label{2.8}(\lambda_{i+1}-\lambda_{1})\int_{\Omega}\psi_{i}^{2}dv\leq\int_{\Omega}\Big[(-\Delta)^{s}|_{\Omega}(h_{i}u_{1})-h_{i}(-\Delta)^{s}|_{\Omega}u_{1}\Big]\psi_{i}
dv.\end{equation}   It follows from the Green formula \eqref{green-for} that,

\begin{equation*}\begin{aligned}&\quad\int_{\Omega}\Big[(-\Delta)^{s}|_{\Omega}(h_{i}u_{1})-h_{i}(-\Delta)^{s}|_{\Omega}u_{1}\Big]u_{1}dv\\&=\int_{\Omega} h_{i}u_{1}(-\Delta)^{s}|_{\Omega}u_{1}dv-\int_{\Omega} h_{i}u_{1}(-\Delta)^{s}|_{\Omega}u_{1}dv\\&=0,\end{aligned}
\end{equation*} which implies that

\begin{equation}\begin{aligned}\label{2.9}&\quad\int_{\Omega}\Big[(-\Delta)^{s}|_{\Omega}(h_{i}u_{1})-h_{i}(-\Delta)^{s}|_{\Omega}u_{1}\Big]\psi_{i}dv\\&
=\int_{\Omega}\Big[(-\Delta)^{s}|_{\Omega}(h_{i}u_{1})-h_{i}(-\Delta)^{s}|_{\Omega}u_{1}\Big]h_{i}u_{1}dv.\end{aligned}\end{equation} Consequently,
plugging \eqref{2.9} into \eqref{2.8}, one has

\begin{equation}\label{2.10}(\lambda_{i+1}-\lambda_{1})\int_{\Omega}\psi_{i}^{2}dv\leq\int_{\Omega}\Big[(-\Delta)^{s}|_{\Omega}(h_{i}u_{1})-h_{i}(-\Delta)^{s}|_{\Omega}u_{1}\Big]h_{i}u_{1}dv.\end{equation}
By  the formula of partial integration, we infer that
\begin{equation}\label{2.11}\int_{\Omega} u_{1,i}u_{1}dv=0,
\end{equation}
\begin{equation}\label{2.12}2\int_{\Omega} u_{1}\frac{\partial (h_{i}u_{1})}{\partial x_{i}}dv
=\int_{\Omega}\left[\frac{\partial(h_{i}u_{1})}{\partial x_{i}}-h_{i}\frac{\partial}{\partial x_{i}}u_{1}\right]u_{1}dv,\end{equation}
and

\begin{equation}\label{u-psi-ineq}-2\int_{\Omega} u_{1,i}\psi_{i}
dv=2\int_{\Omega} u_{1}\frac{\partial(h_{i}u_{1})}{\partial x_{i}}
dv
+2\tau_{i}\int_{\Omega}
 u_{1,i}u_{1}dv.
\end{equation} Hence, utilizing \eqref{2.11}, \eqref{2.12} and \eqref{u-psi-ineq}, we obtain

\begin{equation}\label{2.13}\int_{\Omega}\left[\frac{\partial}{\partial x_{i}}(h_{i}u_{1})-h_{i}\frac{\partial}{\partial x_{i}}u_{1}\right]u_{1}dv=-2\int_{\Omega}\psi_{i}u_{1,i}dv.
\end{equation} Multiplying both sides of \eqref{2.13} by $(\lambda_{i+1}-\lambda_{1})^{\frac{1}{2}}$ and noticing \eqref{2.10}, one can infer that

\begin{equation}\begin{aligned}\label{2.14}&(\lambda_{i+1}-\lambda_{1})^{\frac{1}{2}}\int_{\Omega}\left[\frac{\partial}{\partial x_{i}}(h_{i}u_{1})-h_{i}\frac{\partial}{\partial x_{i}}u_{1}\right]u_{1}dv\\
&=-2(\lambda_{i+1}-\lambda_{1})^{\frac{1}{2}}\int_{\Omega}\psi_{i}
u_{1,i}dv
\\&\leq\delta(\lambda_{i+1}-\lambda_{1})\int_{\Omega}\psi_{i}^{2}dv+\frac{1}{\delta}
\int_{\Omega}u_{1,i}^{2}dv
\\&\leq \delta\int_{\Omega}\Big[(-\Delta)^{s}|_{\Omega}(h_{i}u_{1})-h_{i}(-\Delta)^{s}|_{\Omega}u_{1}\Big]h_{i}u_{1}dv+\frac{1}{\delta}
\int_{\Omega}u_{1,i}^{2}dv,
\end{aligned}\end{equation} where $\delta$ is a positive constant.
Taking sum on $i$ from $1$ to $n$ in \eqref{2.14}, we get
\begin{equation}\begin{aligned}\label{2.15}&\quad\delta^{2}\sum^{n}_{i=1}\int_{\Omega}\Big[(-\Delta)^{s}|_{\Omega}(h_{i}u_{1})-h_{i}(-\Delta)^{s}|_{\Omega}u_{1}\Big]h_{i}u_{1}dv\\&\quad-\delta\sum^{n}
_{i=1}(\lambda_{i+1}-\lambda_{1})^{\frac{1}{2}}\int_{\Omega}
\left[\frac{\partial}{\partial x_{i}}(h_{i}u_{1})-h_{i}\frac{\partial}{\partial x_{i}}u_{1}\right]u_{1}dv+\sum^{n}_{i=1}\int_{\Omega}u_{1,i}^{2}dv\\&\geq0.\end{aligned}\end{equation}
Here, the left-hand side of \eqref{2.15} is a quadratic polynomial of
$\delta$.  By direct calculation, we  show that

\begin{equation}\begin{aligned}\label{ineq-D-hu}&2\int_{\Omega}\Big[(-\Delta)^{s}|_{\Omega}(h_{i}u_{1})-h_{i}(-\Delta)^{s}|_{\Omega}u_{1}\Big]h_{i}u_{1}dv
\\&=\int_{\Omega}h_{i}\left[(-\Delta)^{s}|_{\Omega}(h_{i}u_{1})-h_{i}(-\Delta)^{s}|_{\Omega}u_{1}\right] u_{1}dv\\&\quad-  \int_{\Omega}\left[(-\Delta)^{s}|_{\Omega}(h_{i}^{2} u_{1})-h_{i}(-\Delta)^{s}|_{\Omega}(h_{i} u_{1})\right] u_{1}dv\\&=\int_{\Omega}\left[2 h_{i}(-\Delta)^{s}|_{\Omega}(h_{i}u_{1})-h^{2}_{i}(-\Delta)^{s}|_{\Omega}u_{1}-(-\Delta)^{s}|_{\Omega}(h_{i}^{2} u_{1})\right] u_{1}dv
.\end{aligned}\end{equation}
Combining \eqref{2.15} with \eqref{ineq-D-hu}, we get

\begin{equation}\begin{aligned}\label{2.15-1}&\frac{\delta^{2}}{2}\sum^{n}_{i=1}\int_{\Omega}\left[2 h_{i}(-\Delta)^{s}|_{\Omega}(h_{i}u_{1})-h^{2}_{i}(-\Delta)^{s}|_{\Omega}u_{1}-(-\Delta)^{s}|_{\Omega}(h_{i}^{2} u_{1})\right] u_{1}dv\\&-\delta\sum^{n}
_{i=1}(\lambda_{i+1}-\lambda_{1})^{\frac{1}{2}}\int_{\Omega}
\left[\frac{\partial}{\partial x_{i}}(h_{i}u_{1})-h_{i}\frac{\partial}{\partial x_{i}}u_{1}\right]u_{1}dv+\sum^{n}_{i=1}\int_{\Omega}u_{1,i}^{2}dv\\&\geq0.\end{aligned}\end{equation}From \eqref{2.8},\eqref{2.9}  and \eqref{ineq-D-hu}, we can conclude that

$$\int_{\Omega}\left[2 h_{i}(-\Delta)^{s}|_{\Omega}(h_{i}u_{1})-h^{2}_{i}(-\Delta)^{s}|_{\Omega}u_{1}-(-\Delta)^{s}|_{\Omega}(h_{i}^{2} u_{1})\right] u_{1}dv\geq0,$$ for $i=1,\cdots,n$,
which means that its discriminant must be non-positive.
Hence, from \eqref{Theta-1}, \eqref{Theta-2}, \eqref{Psi}, and \eqref{2.15-1}, we
yield \eqref{lem-ineq}. Therefore, we finish the proof of Lemma \ref{lem2.1}.\end{proof}

\section{Proof of Theorem \ref{thm1.1} }\label{sec3}
In this section, we shall give the proof of Theorem \ref{thm1.1} by applying Lemma \ref{lem2.1} proved in the previous section.
Firstly, we need a result proved by Hook \cite{Hook2} in 1990 as follows.

\begin{prop}\label{prop3.1}{\rm(See~ \cite[Theorem 1]{Hook2})}~Let $\mathcal{H}$ be a real or complex inner product space. Let $\mathcal{M}$ be a linear submanifold of $\mathcal{H}$ and $\mathcal{L}: \mathcal{M} \rightarrow \mathcal{H}$ be a linear operator in $\mathcal{H}$. Suppose $l$ is a positive integer and $u$ is a fixed vector such that for all integers $0 \leq r \leq q \leq l$

\begin{equation*}
\left|\left\langle \mathcal{L}^{q} u, u\right\rangle\right|= |\left\langle\mathcal{L}^{q-r} u,\mathcal{L}^{r} u\right\rangle|.
\end{equation*}
Then, for all integers $0 \leq r \leq q \leq l$, when $q$ is even, we have

\begin{equation}\label{Lruu}
\left|\left\langle\mathcal{L}^{r} u, u\right\rangle\right| \leq\left|\left\langle\mathcal{L}^{q} u, u\right\rangle\right|^{r / q}\langle u, u\rangle^{1-r / q}.
\end{equation}
This inequality is also true for $q$ odd and $0 \leq r \leq q \leq l$. Moreover, there is a finite collection of operators $\left\{\mathcal{D}_{j}\right\}_{j=1}^{N}$ in $\mathcal{H}$ such that

\begin{equation}\label{Lquu}
\left|\left\langle\mathcal{L}^{q} u,u\right\rangle\right|=\left|\sum_{j=1}^{N}\left\langle\mathcal{D}_{j}\mathcal{L}^{q-r} u,\mathcal{D}_{j}\mathcal{L}^{r-1} u\right\rangle\right|.
\end{equation}
holds for each pair of integers $r$ and $q$ with $1 \leq r \leq q \leq l$.\end{prop}

\begin{rem}{\rm(See \cite[Remark (ii)]{Hook2} or  \cite[Remark 2.2]{CZ})}\label{rem-2.1}  If $\mathcal{L}$ is any (not necessarily bounded) nonnegative self-adjoint operator, then \eqref{Lquu} is satisfied by taking $N=1$ and letting $\mathcal{D} _{1}$ be any square root of $\mathcal{L}$. Hence, for nonnegative self-adjoint operators, the inequality \eqref{Lruu} holds for all integers $q$.\end{rem}
By Proposition \ref{prop3.1}, we shall prove the following lemma.

\begin{lem}\label{lem3.1}Assume that the function $\Theta_{2}$ is given by \eqref{Theta-2}.Then, under the same condition as Theorem {\rm \ref{thm1.1}}, we have
\begin{equation}\begin{aligned}\label{Dalta-uu}
\sum_{i=1}^{n}\int_{\Omega}\Theta_{2}(x_{i}) u_{1}dv=s(n+2 s-2) \lambda_{1}^{\frac{s-1}{s}},
\end{aligned}\end{equation}where  $x_{1}, x_{2}, \cdots , x_{n}$ are $n$ arbitrary  coordinate functions on $\mathbb{R}^{n}$.\end{lem}

\begin{proof}
We consider the operator $\left.(-\Delta)^{s}\right|_{\Omega}$ and restrict it to functions supported within bounded open domain $\Omega \subset\mathbb{R}^{n}$. In what follows, some calculations are direct, and we refer the reader to \cite{CZ}. In addition, some more detailed calculations also could be found in \cite{WSZ}.  For example,  by the definition of fractional Laplacian and utilizing some properties of Fourier transformation, we can prove

\begin{equation}
\begin{aligned}\label{ineq-3-det}
&\quad\Big((-\Delta)^{s}|_{\Omega} x_{i}-x_{i} (-\Delta)^{s}|_{\Omega}\Big) u_{1}\\
&=\chi_{\Omega}\mathscr{F}^{-1}\left[|y|^{2 s}\mathscr{F}\left(x_{i} u_{1}\right)\right]-\chi_{\Omega} x_{i}\mathscr{F}^{-1}\left[|y|^{2 s} \hat{u}_{1}\right] \\
&=\sqrt{-1} \chi_{\Omega}\mathscr{F}^{-1}\left[|y|^{2 s} \frac{\partial \hat{u}_{1}}{\partial  y _{i}}-\frac{\partial}{\partial  y _{i}}\left(|y|^{2 s} \hat{u}_{1}\right)\right] \\
&=-\sqrt{-1} \chi_{\Omega}\mathscr{F}^{-1}\left[2 s|y|^{2 s-2}  y _{i} \hat{u}_{1}\right],
\end{aligned}
\end{equation}where $u_{1}$ denotes the first eigenfunction of $\left.(-\Delta)^{s}\right|_{\Omega}$.
According to \eqref{Theta-2} and the similar calculation as \eqref{ineq-3-det}, one can infer by
putting $h_{i}=x_{i}$  that

\begin{equation}\begin{aligned}\label{xax}
&\quad2\Theta_{2}(x_{i})\\&=2 x_{i} (-\Delta)^{s}|_{\Omega} (x_{i} u_{1})-x_{i}^{2} (-\Delta)^{s}|_{\Omega} u_{1}-(-\Delta)^{s}|_{\Omega} (x_{i}^{2} u_{1}) \\
&=2 \chi_{\Omega} x_{i}\mathscr{F}^{-1}\left[|y|^{2 s}\mathscr{F}\left(x_{i} u_{1}\right)\right]-\chi_{\Omega} x_{i}^{2}\mathscr{F}^{-1}\left(|y|^{2 s} \hat{u}_{1}\right)-\chi_{\Omega}\mathscr{F}^{-1}\left[|y|^{2 s}\mathscr{F}\left(x_{i}^{2} u_{1}\right)\right] \\
&=\chi_{\Omega}\left[2\sqrt{-1}x_{i}\mathscr{F}^{-1}\left(|y|^{2 s} \frac{\partial \hat{u}_{1}}{\partial  y _{i}}\right)-x_{i}^{2}\mathscr{F}^{-1}\left(|y|^{2 s} \hat{u}_{1}\right)+\mathscr{F}^{-1}\left(|y|^{2 s} \frac{\partial^{2} \hat{u}_{1}}{\partial  y _{i}^{2}}\right)\right] \\
&=\chi_{\Omega}\mathscr{F}^{-1}\left[-2 \frac{\partial}{\partial  y _{i}}\left(|y|^{2 s} \frac{\partial \hat{u}_{1}}{\partial  y _{i}}\right)+\frac{\partial^{2}}{\partial y _{i}^{2}}\left(|y|^{2 s} \hat{u}_{1}\right)+|y|^{2 s} \frac{\partial^{2} \hat{u}_{1}}{\partial  y _{i}^{2}}\right] \\
&=\chi_{\Omega}\mathscr{F}^{-1}\left[\left(2 s|y|^{2 s-2}+2 s(2 s-2)|y|^{2 s-4}  y _{i}^{2}\right) \hat{u}_{1}\right].
\end{aligned}\end{equation}
Since the support of $u_{1}$ is in $\Omega,$ from \eqref{xax}, we have
\[
\begin{aligned}
\sum_{i=1}^{n}\int_{\Omega}\Theta_{2}(x_{i}) u_{1}dv &=\frac{1}{2} \sum_{i=1}^{n}\int_{\Omega}\left[2 x_{i} (-\Delta)^{s}|_{\Omega}( x_{i} u_{1})-x_{i}^{2} (-\Delta)^{s}|_{\Omega} u_{1}-(-\Delta)^{s}|_{\Omega}( x_{i}^{2} u_{1})\right] u_{1}dv \\
&=\frac{1}{2} \sum_{i=1}^{n}\int_{\Omega}\mathscr{F}^{-1}\left[\left(2 s|y|^{2 s-2}+2 s(2 s-2)|y|^{2 s-4}  y _{i}^{2}\right) \hat{u}_{1}\right]u_{1}dv \\
&=\sum_{i=1}^{n}\int_{\Omega}\left(s|y|^{2 s-2}+s(2 s-2)|y|^{2 s-4}  y _{i}^{2}\right) \hat{u}_{1}^{2}dv \\
&=s(n+2 s-2)\int_{\Omega} |y|^{2 s-2} \hat{u}_{1}^{2}dv \\
&=s(n+2 s-2)\int_{\Omega}u_{1} \mathscr{F}^{-1}\left(|y|^{2 s-2} \hat{u}_{1}\right)dv \\
&=s(n+2 s-2)\int_{\Omega}u_{1}(-\Delta)^{s-1} u_{1}  dv.
\end{aligned}
\]
Next, recall that $\mathbb{Q}_{+}$
is the set of all of the positive rational numbers and consider the following three cases.

\noindent \textbf{Case I.}\quad We assume that $s=\frac{1}{m}, ~m\geq2$, where $m$ is a positive integer. Then, we have
\[
\begin{aligned}
\int_{\Omega}u_{1}(-\Delta)^{s-1} u_{1} dv &=\int_{\Omega}u_{1}\left[(-\Delta)^{-\frac{1}{m}}\right]^{m-1} u_{1} dv\\&=\int_{\Omega}u_{1}\left(((-\Delta)^{s})^{-1}\right)^{m-1} u_{1} dv \\
&=\left(\lambda_{1}^{-1}\right)^{m-1}\\&=\lambda_{1}
^{\frac{s-1}{s}}.
\end{aligned}
\]

\noindent \textbf{Case II.}\quad We assume that $\frac{1}{2}<s<1,$ where $s \in \mathbb{Q}_{+}$, and denote $s=\frac{a_{2}}{a_{1}},$ where $a_{1}, a_{2}$ are positive integers and $a_{2}
<a_{1}<2 a_{2} .$ Meanwhile, putting

\begin{equation}\label{opera-T}\mathcal{L} =(-\Delta)^{-\frac{1}{a_{1}}},\end{equation} it is well known that it is a positive and self-adjoint operator. Thus, the operator $\mathcal{L}$ given by \eqref{opera-T} satisfies the condition in Proposition \ref{prop3.1}. Furthermore, by Remark \ref{rem-2.1}, then one can show that inequality \eqref{Lruu} is valid for all $0 \leq r \leq q \leq l$ without parity condition on $q .$ Applying $0<r=a_{1}-a_{2}<a_{2}$ and $q=a_{2}
$ to \eqref{Lruu}, we have
\[
\begin{aligned}
\int_{\Omega}u_{1}(-\Delta)^{s-1} u_{1} dv &=\int_{\Omega}u_{1}\left[(-\Delta)^{-\frac{1}{a_{1}}}\right]^{a_{1}-a_{2}} u_{1} dv \\
& \leq\left[\int_{\Omega}u_{1}\left[(-\Delta)^{-\frac{1}{a_{1}}}\right]^{a_{2}} u_{1} dv\right]^{\frac{a_{1}-a_{2}}{a_{2}}}\left[\int_{\Omega} u_{1}^{2} dv\right]^{1-\frac{a_{1}-a_{2}}{a_{2}}} \\
&=\left(\lambda_{1}^{-1}\right)^{\frac{a_{1}-a_{2}}{a_{2}}}\\&=\lambda_{1}^{\frac{s-1}{s}}.
\end{aligned}
\]

\noindent \textbf{Case III.}\quad We assume that $s \geq 1$, where $s \in \mathbb{Q}_{+}$, and  denote $s=\frac{a_{2}}{a_{1}},$ where $a_{1}, a_{2}$ are positive integers and $a_{2} \geq a_{1}$. In Proposition \ref{prop3.1}, we shall choose an operator $\mathcal{L} $ satisfying
\begin{equation*} \mathcal{L} =(-\Delta)^{\frac{1}{a_{1}}}.\end{equation*} Then, $\mathcal{L}$ is a positive and self-adjoint operator. According to Remark \ref{rem-2.1}, we know that inequality \eqref{Lruu} is true for all $0 \leq r \leq q \leq l$ without parity condition on $q$. Applying \eqref{Lruu} with $0 \leq r=a_{2}-a_{1}<a_{2}$ and $q=a_{2},$ we can deduce that
\[
\begin{aligned}
\int_{\Omega}u_{1}(-\Delta)^{s-1} u_{1} dv &=\int_{\Omega}u_{1}\left[(-\Delta)^{\frac{1}{a_{1}}}\right]^{a_{2}-a_{1}} u_{1} dv \\
& \leq\left[\int_{\Omega}u_{1}\left[(-\Delta)^{\frac{1}{a_{1}}}\right]^{a_{2}} u_{1} dv\right]^{\frac{a_{2}-a_{1}}{a_{2}}}\left[\int_{\Omega} u_{1}^{2} dv\right]^{1-\frac{a_{2}-a_{1}}{a_{2}}}\\
&=\lambda_{1}^{\frac{a_{2}
-a_{1}}{a_{2}}}\\&=\lambda_{1}^{\frac{s-1}{s}}.
\end{aligned}
\]
Hence, for all $s=\frac{1}{m}$ with the positive integer $m \geq 2$ or $s>\frac{1}{2}$ with $s \in \mathbb{Q}_{+},$ we have

\begin{equation}\begin{aligned}
\sum_{i=1}^{n}\int_{\Omega}\Theta_{2}(x_{i}) u_{1}dv&=s(n+2 s-2)\int_{\Omega}u_{1}(-\Delta)^{s-1} u_{1} dv\\&\leq s(n+2 s-2) \lambda_{1}^{\frac{s-1}{s}}.
\end{aligned}\end{equation}
Therefore, we finish the proof  of Lemma \ref{lem3.1}.
\end{proof}
In what follows, we assume that $x_{1},x_{2},\cdots,x_{n}$ are $n$ coordinate functions on $\mathbb{R}^{n}$.
In order to make use of  Lemma \ref{lem2.1}, we construct some functions satisfying \eqref{ortho} by using $\left\{x_{i}\right\}_{i=1}^{ n} .$ We consider an $ n \times n$ matrix

$$ Q=\left(\int_{\Omega} x_{i} u_{1} u_{j+1}dv\right)_{ n \times n}.$$ According
to the QR-factorization theorem, we know that there exists an orthogonal $ n \times n$  matrix $ P=\left(p_{i j}\right)_{ n \times n},$ such that $U=P Q$ is an upper triangle matrix. Namely, we have
\[
\sum_{k=1}^{ n} p_{i k} \int_{\Omega} x_{k} u_{1} u_{j+1}dv=0, \quad \text { for } 1 \leq j<k \leq  n.
\]
Define functions $\overline{x}_{i}$ by

\begin{equation}\label{coor-func}
\overline{x}_{i}=\sum_{k=1}^{ n} p_{i k} x_{k}.
\end{equation}Thus we infer

\begin{equation}
\int_{\Omega} \overline{x}_{i} u_{1} u_{j+1}dv=0, \quad \text { for } 1 \leq j<i \leq  n.
\end{equation}
Moreover, because $P$ is an orthogonal matrix,
by the orthogonality, we infer that

\begin{equation} \begin{aligned} \label{le-ineq-1}  \sum^{n}_{i=1}\overline{x}_{i}(-\Delta)^{s}|_{\Omega}(\overline{x}_{i}u_{1})&= \sum^{n}_{i=1}\left(\sum_{k=1}^{ n} p_{ik}x_{k}\right)(-\Delta)^{s}|_{\Omega}\left( \sum_{l=1}^{n} p_{il} x_{l}u_{1}\right)\\
&=\sum^{n}_{i=1}\left(\sum_{k=1}^{ n} p_{ik}x_{k}\right)\chi_{\Omega}\mathscr{F}^{-1}\left[|y|^{2s} \mathscr{F}\left[\left( \sum_{l=1}^{n} p_{il} x_{l}u_{1}\right)\right]\right]\\
&=\sum^{n}_{i=1}\left(\sum_{k=1}^{ n} p_{ik}x_{k}\right)\sum_{l=1}^{n} p_{il}\left\{\chi_{\Omega}\mathscr{F}^{-1}\left[|y|^{2s}  \mathscr{F}\left[\left( x_{l}u_{1}\right)\right]\right]\right\}\\
&=\sum^{n}_{i=1}\left(\sum_{k=1}^{ n} p_{ik}x_{k}\right)\sum_{l=1}^{n} p_{il}(-\Delta)^{s}|_{\Omega}(x_{l}u_{1})\\
&=\sum_{k,l=1}^{ n}\left(\sum^{n}_{i=1} p_{ik}p_{il}\right)x_{k} (-\Delta)^{s}|_{\Omega}(x_{l}u_{1})\\
&=  \sum^{n}_{i=1}  x_{i} (-\Delta)^{s}|_{\Omega}(x_{i}u_{1}).\end{aligned}\end{equation}Similarly, we can prove that

\begin{equation} \begin{aligned}  \label{le-ineq-2}  \sum^{n}_{i=1}\overline{x}^{2}_{i}(-\Delta)^{s}|_{\Omega}u_{1}=\sum^{n}_{i=1}x^{2}_{i}(-\Delta)^{s}|_{\Omega}u_{1},\end{aligned}\end{equation}
and

\begin{equation} \begin{aligned}\label{le-ineq-3}   \sum^{n}_{i=1}(-\Delta)^{s}|_{\Omega}(\overline{x}
_{i}^{2} u_{1})=\sum^{n}_{i=1}(-\Delta)^{s}|_{\Omega}(x_{i}^{2} u_{1}).\end{aligned}\end{equation}
From \eqref{le-ineq-1}, \eqref{le-ineq-2}, \eqref{le-ineq-3} and \eqref{Dalta-uu}, we have
\[
\begin{aligned}
\sum_{i=1}^{n}\int_{\Omega}\Theta_{2}(\overline{x}_{i}) u_{1}dv
&=s(n+2 s-2) \lambda_{1}^{\frac{s-1}{s}}.
\end{aligned}
\]
Therefore, based on the above argument, we have the following lemma.

\begin{lem}\label{lem3.2}Assume that the function $\Theta_{2}$ is given by \eqref{Theta-2}.Then, under the same condition as Theorem {\rm \ref{thm1.1}}, we have
\begin{equation}\begin{aligned}\label{Dalta-uu-1}
\sum_{i=1}^{n}\int_{\Omega}\Theta_{2}(\overline{x}_{i}) u_{1}dv=s(n+2 s-2) \lambda_{1}^{\frac{s-1}{s}},
\end{aligned}\end{equation}where  $\overline{x}_{1}, \overline{x}_{2}, \cdots , \overline{x}_{n}$ are $n$ Euclidean coordinate functions satisfying \eqref{coor-func}.\end{lem}
$$\eqno\Box$$
Next, applying Lemma \ref{lem2.1} and Lemma \ref{lem3.2}, we shall give the proof of Theorem \ref{thm1.1}.
\vskip 3mm
\noindent \emph{Proof of Theorem} \ref{thm1.1}.
Assume that $\{x_{1}, x_{2}, \cdots , x_{n}\}$ is a coordinate system given in Lemma \ref{lem3.1}   on $\mathbb{R}^{n}$. After a rotation transformation for the coordinate system  $\{x_{1}, x_{2}, \cdots , x_{n}\}$, we  can construct a new coordinate system $\{\overline{x}_{1}, \overline{x}_{2}, \cdots , \overline{x}_{n}\}$  satisfying equation \eqref{coor-func}.
Furthermore, in \eqref{lem-ineq}, we take $h_{i}=\overline{x}_{i}$, where $i=1,2,\cdots,n$. Then, under the new coordinate system $\{\overline{x}_{1}, \overline{x}_{2}, \cdots , \overline{x}_{n}\}$,  we have

\begin{equation}\label{lem-ineq-2}\sum^{n}_{i=1}(\lambda_{i+1}-\lambda_{1})^{\frac{1}{2}}\int_{\Omega}\Theta_{1}(\overline{x}_{i})u_{1}dv
\leq2\left\{\sum^{n}_{i=1}\int_{\Omega}\Theta_{2}(\overline{x}_{i})u_{1}dv\sum^{n}_{i=1}\int_{\Omega}\Psi_{i} dv\right\}^{\frac{1}{2}},\end{equation}and
$\int_{\Omega}\overline{x}_{i}u_{1}u_{j+1}dv=0, \ \ {\rm for}~ {\rm all} \ \ 0\leq j<i\leq n,
$ which happens to satisfy condition \eqref{ortho} in Lemma \ref{lem2.1}. From \eqref{Theta-1}, we have

\begin{equation*}\begin{aligned}\int_{\Omega}\Theta_{1}(\overline{x}_{i})u_{1}dv&=\int_{\Omega} \left[\frac{\partial}{\partial \overline{x}_{i}}(\overline{x}_{i}u_{1})-\overline{x}_{i}\frac{\partial}{\partial \overline{x}_{i}}u_{1} \right]u_{1}dv\\&=\int_{\Omega} \left[u_{1}\frac{\partial}{\partial \overline{x}_{i}}(\overline{x}_{i})+\overline{x}_{i}\frac{\partial}{\partial \overline{x}_{i}}u_{1}-\overline{x}_{i}\frac{\partial}{\partial \overline{x}_{i}}u_{1}\right]u_{1}dv\\&=\int_{\Omega}u_{1}^{2}dv=1,\end{aligned}\end{equation*}which implies that

\begin{equation}\label{Theta-1-estimate}\sum^{n}_{i=1}(\lambda_{i+1}-\lambda_{1})^{\frac{1}{2}}\int_{\Omega}\Theta_{1}(\overline{x}_{i})u_{1}dv
=\sum^{n}_{i=1}(\lambda_{i+1}-\lambda_{1})^{\frac{1}{2}}.\end{equation}
Moreover, we suppose $s>0$ with $s \in \mathbb{Q}_{+},$ and denote $s=\frac{a_{2}}{a_{1}}$ as before, where $a_{1}, a_{2}$ are two positive integers. Furthermore, we assume that the operator $\mathcal{L}$ satisfies that $\mathcal{L}=-\Delta.$ Then, it is clear that the conditions of Proposition \ref{prop3.1} are obviously satisfied by the above choice of the operator $\mathcal{L}$.   Therefore,  for all $0 \leq r \leq q \leq l$, inequality \eqref{Lruu} is true without parity condition on $q .$ Applying inequality \eqref{Lruu} with $r=1$ and $q=a_{2},$ it is not difficult to give an upper bound for the function   $\Psi_{i}$ as follows:

\begin{equation}
\begin{aligned}\label{sum-Tu}\sum^{n}_{i=1}\int_{\Omega}\Psi_{i} dv&=
\int_{\Omega}u_{1}(-\Delta) u_{1}dv \\
& \leq\left\{\int_{\Omega}u_{1}(-\Delta)^{a_{2}} u_{1}
dv\right\}^{\frac{1}{a_{2}}}\left\{\int_{\Omega} u_{1}^{2} dv\right\}^{1-\frac{1}{a_{2}}}\\&=\left\{\int_{\Omega}u_{1}\left((-\Delta)^{\frac{a_{2}}{a_{1}}}\right)^{a_{1}} u_{1} dv\right\}^{\frac{1}{a_{2}}} \\
&=\left\{\int_{\Omega}u_{1}\Big(\left(-\Delta\right)^{s}\Big)^{a_{1}} u_{1}  dv\right\}^{\frac{1}{a_{2}}}\\&=\lambda_{1}^{\frac{a_{1}}{a_{2}}}=\lambda_{1}^{\frac{1}{s}}.
\end{aligned}
\end{equation}
Substituting \eqref{Dalta-uu-1}, \eqref{Theta-1-estimate} and \eqref{sum-Tu} into \eqref{lem-ineq-2}, we get inequality \eqref{thm1.1-ineq}. Therefore, we finish the proof of Theorem \ref{thm1.1}.
$$\eqno\Box$$

\begin{ack} The author will express his sincere gratitude to the anonymous referees for their helpful comments and suggestions. Many thanks to professor Huyuan Chen for helpful discussion and drawing my attention to his article {\rm \cite{CV}}.\end{ack}

\end{document}